\documentclass[12pt]{amsart}

\usepackage{graphicx}
\usepackage{amssymb}
\usepackage{amsthm}
\usepackage{amsmath}
\usepackage{stmaryrd}
\usepackage{cite}
\usepackage[mathscr]{eucal}
\usepackage{hyperref}
\usepackage[margin=1.0in]{geometry}
\usepackage{comment}

\usepackage{url}
\usepackage{hyperref}

\title[Betti numbers and secondary curvature]{Betti numbers and the curvature operator of the second kind}
\author[Nienhaus, Petersen and Wink]{Jan Nienhaus, Peter Petersen and Matthias Wink}
\address{Department of Mathematics, UCLA, 520 Portola Plaza, Los Angeles, CA, 90095}
\email{petersen@math.ucla.edu}

\address{Mathematisches Institut, Universit\"at M\"unster, Einsteinstra{\ss}e 62, 48149 M\"unster}
\email{j.nienhaus@uni-muenster.de}
\email{mwink@uni-muenster.de}


\makeatletter
\@namedef{subjclassname@2020}{
\textup{2020} Mathematics Subject Classification}
\makeatother

\thanks{JN and MW funded by the Deutsche Forschungsgemeinschaft (DFG, German Research Foundation) under Germany's Excellence Strategy EXC 2044–390685587, Mathematics M\"unster: Dynamics–Geometry–Structure.}

\begin{document}
\newcommand{\Ext}{\bigwedge\nolimits}
\newcommand{\Div}{\operatorname{div}}
\newcommand{\Hol} {\operatorname{Hol}}
\newcommand{\diam} {\operatorname{diam}}
\newcommand{\Scal} {\operatorname{Scal}}
\newcommand{\scal} {\operatorname{scal}}
\newcommand{\Ric} {\operatorname{Ric}}
\newcommand{\Hess} {\operatorname{Hess}}
\newcommand{\grad} {\operatorname{grad}}
\newcommand{\Sect} {\operatorname{Sect}}
\newcommand{\Rm} {\operatorname{Rm}}
\newcommand{ \Rmzero } {\mathring{\Rm}}
\newcommand{\Rc} {\operatorname{Rc}}
\newcommand{\Curv} {S_{B}^{2}\left( \mathfrak{so}(n) \right) }
\newcommand{ \tr } {\operatorname{tr}}
\newcommand{ \id } {\operatorname{id}}
\newcommand{ \Riczero } {\mathring{\Ric}}
\newcommand{ \ad } {\operatorname{ad}}
\newcommand{ \Ad } {\operatorname{Ad}}
\newcommand{ \dist } {\operatorname{dist}}
\newcommand{ \rank } {\operatorname{rank}}
\newcommand{\Vol}{\operatorname{Vol}}
\newcommand{\dVol}{\operatorname{dVol}}
\newcommand{ \zitieren }[1]{ \hspace{-3mm} \cite{#1}}
\newcommand{ \pr }{\operatorname{pr}}
\newcommand{\diag}{\operatorname{diag}}
\newcommand{\Lagr}{\mathcal{L}}
\newcommand{\av}{\operatorname{av}}
\newcommand{ \floor }[1]{ \lfloor #1 \rfloor }
\newcommand{ \ceil }[1]{ \lceil #1 \rceil }
\newcommand{\Sym} {\operatorname{Sym}}
\newcommand{\bcirc}{ \ \bar{\circ} \ }

\newtheorem{theorem}{Theorem}[section]
\newtheorem{definition}[theorem]{Definition}
\newtheorem{example}[theorem]{Example}
\newtheorem{remark}[theorem]{Remark}
\newtheorem{lemma}[theorem]{Lemma}
\newtheorem{proposition}[theorem]{Proposition}
\newtheorem{corollary}[theorem]{Corollary}
\newtheorem{assumption}[theorem]{Assumption}
\newtheorem{acknowledgment}[theorem]{Acknowledgment}
\newtheorem{DefAndLemma}[theorem]{Definition and lemma}

\newenvironment{remarkroman}{\begin{remark} \normalfont }{\end{remark}}
\newenvironment{exampleroman}{\begin{example} \normalfont }{\end{example}}

\newcommand{\R}{\mathbb{R}}
\newcommand{\N}{\mathbb{N}}
\newcommand{\Z}{\mathbb{Z}}
\newcommand{\Q}{\mathbb{Q}}
\newcommand{\C}{\mathbb{C}}
\newcommand{\F}{\mathbb{F}}
\newcommand{\X}{\mathcal{X}}
\newcommand{\D}{\mathcal{D}}
\newcommand{\Cont}{\mathcal{C}}

\renewcommand{\labelenumi}{(\alph{enumi})}
\newtheorem{maintheorem}{Theorem}[]
\renewcommand*{\themaintheorem}{\Alph{maintheorem}}
\newtheorem*{theorem*}{Theorem}
\newtheorem*{corollary*}{Corollary}
\newtheorem*{remark*}{Remark}
\newtheorem*{example*}{Example}
\newtheorem*{question*}{Question}
\newtheorem*{definition*}{Definition}

\begin{abstract}
We show that compact, $n$-dimensional Riemannian manifolds with $\frac{n+2}{2}$-nonnegative curvature operators of the second kind are either rational homology spheres or flat. 

More generally, we obtain vanishing of the $p$-th Betti number provided that the curvature operator of the second kind is $C(p,n)$-positive. Our curvature conditions become weaker as $p$ increases. For $p=\frac{n}{2}$ we have $C(p,n)= \frac{3n}{2} \frac{n+2}{n+4} $, and for $5 \leq p \leq \frac{n}{2}$ we exhibit a $C(p,n)$-positive algebraic curvature operator of the second kind with negative Ricci curvatures. 
\end{abstract}

\maketitle

\section*{Introduction}

It is an important topic in geometry to understand how geometric assumptions restrict the topology of the underlying Riemannian manifold. For example, D. Meyer \cite{DMeyerCurvOpPos} showed that manifolds with positive curvature operators are rational homology spheres. Gallot-Meyer \cite{GallotMeyerCurvOperatorAndForms} proved the corresponding rigidity theorem. That is, manifolds with nonnegative curvature operators are either reducible, locally symmetric or their universal cover has the cohomology of a sphere or a complex projective space. 

With Ricci flow techniques, these results were improved to diffeomorphism classifications. In particular, due to the work of Hamilton \cite{Hamilton3DimRF, Hamilton4DimRFposCurvOp}, Chen \cite{ChenQuarterPinching} and B\"ohm-Wilking \cite{BW2}, manifolds with $2$-positive curvature operators are diffeomorphic to space forms. The corresponding rigidity result was obtained by Ni-Wu \cite{NiWuNonnegativeCurvatureOperator}. Generalizations of these results in the context of isotropic curvatures were proven by Brendle-Schoen \cite{BrendleSchoenWeaklyQuarterPinched, BrendleSchoenSphereTheorem} and Brendle \cite{BrendleConvergenceInHihgerDimensions}.

Moreover, the second and third named authors proved vanishing and rigidity theorems for $p$-forms based on the corresponding assumption on the sum of the lowest $(n-p)$ eigenvalues of the curvature operator in \cite{PetersenWinkNewCurvatureConditionsBochner}. 

\vspace{2mm}

In addition to the curvature operator
\begin{align*}
\mathfrak{R} \colon \Lambda^2 TM \to \Lambda^2 TM, \ \left( \mathfrak{R}(\omega) \right)_{ij} = \sum_{k,l} R_{ijkl}  \omega_{kl},
\end{align*}
the curvature tensor of a Riemannian manifold also induces a self-adjoint operator on the space of symmetric $(0,2)$-tensors
\begin{align*}
\overline{R} \colon S^2(TM) \to S^2(TM), \ \left( \overline{R}(h) \right)_{ij} = \sum_{k,l} R_{iklj} h_{kl}.
\end{align*}
The {\em curvature operator of the second kind} is the induced map on the space of trace-free symmetric $(0,2)$-tensors:
\begin{align*}
\mathcal{R} \colon  S^2_0(TM) \to S^2_0(TM), \ \mathcal{R} = \operatorname{pr}_{S^2_0(TM)} \circ \overline{R}_{|S_0^2(TM)}.
\end{align*}

It was already studied by Bourguignon-Karcher in \cite{BourguignonKarcherCurvOperatorsPinchingEstimates}. In contrast to $\overline{R}$, the curvature operator of the second kind $\mathcal{R}$ satisfies the natural geometric condition that $\mathcal{R} \geq \kappa$ implies that all sectional curvatures are bounded from below by $\kappa.$ 

Ogiue-Tachibana \cite{OgiueTachibanaVarietesRiemCourbureRestreint} proved that similarly to D. Meyer's result, compact manifolds with positive curvature operators of the second kind are rational homology spheres. Both proofs rely on the Bochner technique. 

Nishikawa \cite{NishikawaDeformationRiemMetrics} conjectured that compact manifolds with positive curvature operators of the second kind are diffeomorphic to spherical space forms. In \cite{CaoGurskyTranNishikawaConjecture}, Cao-Gursky-Tran proved Nishikawa's conjecture. In fact, they proved Nishikawa's conjecture for manifolds with $2$-positive curvature operators of the second kind. Subsequently, X.Li \cite{LiCurvatureOperatorSecondKind} relaxed the assumption to $3$-positive curvature operator of the second kind. The proofs are based on the observation that these manifolds satisfy the PIC1 condition and thus Brendle's \cite{BrendleConvergenceInHihgerDimensions} convergence result for the Ricci flow applies. 

In addition, the rigidity part of Nishikawa's conjecture \cite{NishikawaDeformationRiemMetrics} asserts that a manifold with nonnegative curvature operator of the second kind is diffeomorphic to a locally symmetric space. In \cite{LiCurvatureOperatorSecondKind}, X.Li proved that a Riemannian manifold of dimension $n \geq 4$ with $3$-nonnegative curvature operator of the second kind is either diffeomorphic to a spherical space form, flat, or $n \geq 5$ and the universal cover is isometric to a compact irreducible symmetric space. 

The first main theorem of the paper rules out the third option, even under a weaker assumption on the eigenvalues of the curvature operator of the second kind.

\begin{definition*}
A self-adjoint operator $\mathcal{R}$ with eigenvalues $\lambda_1 \leq \lambda_2 \leq \ldots \leq \lambda_N$ is called $k$-nonnegative for some $k \geq 1$ if $\lambda_1 + \ldots + \lambda_{\floor{k}} + \left( k - \floor{k} \right) \lambda_{\floor{k}+1} \geq 0.$
\end{definition*}
Note that $\mathcal{R}$ is $k$-nonnegative if it is $\floor{k}$-nonnegative. We say $\mathcal{R}$ is nonnegative if it is $1$-nonnegative.

\begin{maintheorem}
\label{NishikawasConjecture}
Let $(M,g)$ be a compact, $n$-dimensional Riemannian manifold. If the curvature operator of the second kind is $\frac{n+2}{2}$-nonnegative, then $(M,g)$ is either flat or a rational homology sphere. 
\end{maintheorem}

Compact symmetric spaces which are real cohomology spheres were classified by Wolf in \cite{WolfSymmetricRealCohomSpheres}. Apart from spheres, $SU(3)/SO(3)$ is the only simply connected example. However, according to example \ref{CurvatureSU3SO3}, the curvature operator of the second kind of $SU(3)/SO(3)$ is $9$-positive but not $8$-nonnegative. Thus, combining Theorem \ref{NishikawasConjecture} with X.Li's result explained above \cite{LiCurvatureOperatorSecondKind}, we obtain the following improvement on Nishikawa's conjecture. 

\begin{corollary*}
Let $n \geq 4$ and let $(M,g)$ be a compact, $n$-dimensional Riemannian manifold. If the curvature operator of the second kind is  $3$-nonnegative, then $(M,g)$ is either flat or diffeomorphic to a spherical space form. 
\end{corollary*}
In dimension $n=3,$ X.Li proved the above result in \cite{LiCurvatureOperatorSecondKind}.

The proof of Theorem \ref{NishikawasConjecture} is an application of a Bochner formula for the curvature operator of the second kind. In the case of $\frac{n+2}{2}$-nonnegative curvature operator, we are able to obtain control on all Betti numbers. This is also the case for Einstein manifolds:

\begin{maintheorem}
\label{MainTheoremEinstein}
Let $(M,g)$ be a compact, $n$-dimensional Einstein manifold. Let $N = \frac{3n}{2}\frac{n+2}{n+4}.$ 
\begin{enumerate}
\item If the curvature operator of the second kind is $N$-positive, then $M$ is a rational homology sphere. 
\item If the curvature operator of the second kind is $N'$-nonnegative for some $N'<N,$ then $(M,g)$ is either flat or a rational homology sphere.
\item If the curvature operator of the second kind is $N$-nonnegative, then all harmonic forms are parallel. 
\end{enumerate}
\end{maintheorem}

\begin{remark*}
\normalfont
By the theory of Diophantine equations, $\frac{3n}{2}\frac{n+2}{n+4}$ is only an integer if $n=0,2,8.$ 
Therefore, unless $n=2,8,$ if the curvature operator of the second kind is $\floor{N}$-nonnegative, part (b) applies. 

Furthermore, the curvature condition in part (b) implies that either $\mathcal{R}$ is $N$-positive or $1$-nonnegative, cf. theorem \ref{WeightPrinciple} (d). 
\end{remark*}

Theorem \ref{MainTheoremEinstein} amplifies work of Cao-Gursky-Tran \cite{CaoGurskyTranNishikawaConjecture} who proved that Einstein manifolds with $4$-nonnegative curvature operators of the second kind are locally symmetric, and have constant sectional curvature in the case of $4$-positivity. This is a consequence of their observation that manifolds with $4$-nonnegative curvature operators of the second kind have nonnegative isotropic curvature, and Brendle's theorem \cite{BrendleEinsteinNIC} on Einstein manifolds with nonnegative isotropic curvature. Much earlier, Kashiwada \cite{KashiwadaCurvOperatorSecondKind} proved the theorem for manifolds with positive curvature operators of the second kind. 

\vspace{2mm}

For a general Riemannian manifold, we obtain the following vanishing and rigidity results for the $p$-th Betti number. Note that due to Poincar\'e duality we may assume $p \leq \frac{n}{2}.$ Set
\begin{align*}
C_p=C_p(n)=\frac{3}{2}  \frac{n(n+2) p(n-p)}{n^2 p - n p^2 - 2np + 2n^2 + 2n -4p}.
\end{align*}

\begin{maintheorem}
\label{BettiNumbersCurvatureSecondKind}
Let $(M,g)$ be a compact, $n$-dimensional Riemannian manifold and let $p \leq \frac{n}{2}$. 
\begin{enumerate}
\item If the curvature operator of the second kind is $C_p$-positive, then the $p$-th Betti number $b_p(M, \R)$ vanishes. 
\item If the curvature operator of the second kind is $C'$-nonnegative for some $C'<C_p,$ then $b_p(M, \R)$ vanishes or $(M,g)$ is flat.
\item If the curvature operator of the second kind is $C_p$-nonnegative, then all harmonic $p$-forms are parallel. 
\end{enumerate}
\end{maintheorem}

Note that $C_p$ increases as $p \leq \frac{n}{2}$ increases. In particular, the curvature conditions become weaker as $p$-increases. Therefore, unless $(M,g)$ is flat, if the curvature operator of the second kind is $C_p$-nonnegative, then all harmonic $k$-forms vanish for $p < k < n-p.$ Furthermore, we obtain the weakest curvature condition for $p= \frac{n}{2}.$ In this case we have $C_{\frac{n}{2}}=\frac{3n}{2} \frac{n+2}{n+4}$ as in the Einstein case in Theorem \ref{MainTheoremEinstein}. 

The effect that curvature conditions become weaker as $p$ increases also occurs for holomorphic $p$-forms on a compact K\"ahler manifold according to a result of Bochner \cite{BochnerVectorFieldsAndRic}. However, in the case of manifolds with generic holonomy this is a new phenomenon. 

Due to a result of X.Li \cite{LiCurvatureOperatorSecondKind}, $n$-nonnegativity of $\mathcal{R}$ implies that $\Ric \geq \frac{\scal}{n(n+1)} \geq 0.$ However, $\mathcal{R}$ being $(n+1)$-nonnegative does not imply nonnegativity of Ricci curvature, according to example \ref{NoControlOnRic}. 

Notice that for any fixed $p$ we have 
\begin{align*}
C_p(n) \sim \frac{3n}{2} \frac{p}{p+2}
\end{align*}
for large $n$. In particular, asymptotically for large $n$, our curvature conditions do not imply lower Ricci curvature bounds for $p \geq 5.$ Specifically, for $p=5$ we have $C_5(n) \geq \frac{15 n}{14} \geq n+1$ if $n \geq 14$. In contrast, the results for the (standard) curvature operator in \cite{PetersenWinkNewCurvatureConditionsBochner} imply lower Ricci curvature bounds for any $p.$

\vspace{2mm}

In addition to vanishing and rigidity results, our methods also yield estimation results in the presence of lower Ricci curvature bounds. This is a consequence of the techniques developed by Gallot \cite{GallotSobolevEstimates} and P.Li \cite{LiSobolevConstant}. In particular, Gallot proved estimation theorems for the Betti numbers of manifolds with upper diameter bounds and lower bounds on the (standard) curvature operator. The curvature assumption was weakened in \cite{PetersenWinkNewCurvatureConditionsBochner} to a lower bound on the average of the lowest $(n-p)$-eigenvalues of the curvature operator. 

For the curvature operator of the second kind, lemma \ref{ImprovedEstimatePRicci} yields a lower bound on the Ricci curvature provided that the average of the lowest $n$ eigenvalues of the curvature operator of the second kind is bounded from below. Therefore, the techniques of Gallot and Li yield:

\begin{maintheorem}
\label{WeakEstimation}
Let $n \geq 3,$ $D>0$ and $\kappa \leq 0.$ Let $(M,g)$ be a compact, $n$-dimensional Riemannian manifold. Let $\lambda_1 \leq \ldots \leq \lambda_{ \frac{1}{2}(n+2)(n-1)}$ denote the eigenvalues of the curvature operator of the second kind of $(M,g)$. There is $C(n, D \kappa ^2)>0$ such that if $\diam(M) < D$ and 
\begin{align*}
\begin{cases}
\lambda_1 + \ldots + \lambda_{\floor{\frac{n+2}{2}}} + \frac{1}{2} \cdot \lambda_{\floor{\frac{n+2}{2}}+1} & \geq  \frac{n+2}{2}  \cdot \kappa, \text{ if } n \text{ odd,} \\
\lambda_1 + \ldots + \lambda_{\frac{n+2}{2}} & \geq  \frac{n+2}{2}  \cdot \kappa,  \text{ if } n \text{ even,}
\end{cases}
\end{align*}
then 
\begin{align*}
b_p(M) \leq {n \choose p} \exp \left( C \left(n,\kappa D^2 \right) \cdot \sqrt{-\kappa D^2 p( n-p ) } \right).
\end{align*}
In particular, there is $\varepsilon(n) > 0$ such that $\kappa D^2 > - \varepsilon$ implies $b_p(M) \leq {n \choose p}.$
\end{maintheorem}
In the Einstein case, a lower bound on the scalar curvature implies a lower bound on the Ricci curvature. Thus we obtain the bound on the Betti numbers $b_p(M)$ in Theorem \ref{WeakEstimation} for all $p$ provided the Riemannian manifold $(M,g)$ is Einstein, $\diam(M) < D$ and $\lambda_1 + \ldots + \lambda_{\floor{N}} + ( N - \floor{N} ) \cdot \lambda_{\floor{N}+1} \geq  N \kappa,$ where $N=\frac{3n}{2} \frac{n+2}{n+4}.$

In order to obtain an estimation analog of Theorem \ref{BettiNumbersCurvatureSecondKind}, we impose an explicit lower Ricci curvature bound:
\begin{corollary*}
\label{StrongEstimation}
Let $n \geq 3,$ $D>0$ and $\kappa \leq 0.$ Let $(M,g)$ be a compact, $n$-dimensional Riemannian manifold. Let $\lambda_1 \leq \ldots \leq \lambda_{ \frac{1}{2}(n+2)(n-1)}$ denote the eigenvalues of the curvature operator of the second kind of $(M,g)$. There is $C(n, D \kappa ^2)>0$ such that if $\diam(M) < D$, 
\begin{align*}
\Ric \geq (n-1) \kappa \ \text{ and } \ \lambda_1 + \ldots + \lambda_{\floor{C_p}} + ( C_p - \floor{C_p} ) \lambda_{\floor{C_p}+1} \geq C_p \kappa,
\end{align*}
then 
\begin{align*}
b_p(M) \leq {n \choose p} \exp \left( C \left(n,\kappa D^2 \right) \cdot \sqrt{-\kappa D^2 p( n-p ) } \right).
\end{align*}
In particular, there is $\varepsilon(n) > 0$ such that $\kappa D^2 > - \varepsilon$ implies $b_p(M) \leq {n \choose p}.$
\end{corollary*}

The proofs of the main Theorems are based on the Bochner technique. By Hodge theory, every de Rham cohomology class is represented by a harmonic form. If $\omega$ is a harmonic $p$-form, then it satisfies the Bochner formula
\begin{align*}
\Delta \frac{1}{2} | \omega |^2 = | \nabla \omega |^2 + g ( \Ric_L( \omega ), \omega ).
\end{align*}
We establish that the curvature term of the Lichnerowicz Laplacian satisfies the equation 
\begin{align*}
\frac{3}{2} g ( \Ric_L( \omega ), \omega )  = \sum_{\alpha = 1 }^{N} \lambda_{\alpha} | S_{\alpha} \omega |^2 + \frac{p(n-2p)}{n} \sum_{j,k} \sum_{i_2, \ldots, i_p} R_{jk} \omega_{j i_2 \ldots i_p} \omega_{k i_2 \ldots i_p} + \frac{p^2}{n^2} \scal | \omega |^2,
\end{align*}
where $\lbrace S_{\alpha} \rbrace$ is an orthonormal eigenbasis of the curvature operator of the second kind with corresponding eigenvalues $\lbrace \lambda_{\alpha} \rbrace$, and $N= \dim S_0^2(TM) = \frac{1}{2} (n-1)(n+2).$ 

We are able to control the first term by understanding the interaction of trace-free, symmetric tensors on forms, adapting ideas of \cite{PetersenWinkNewCurvatureConditionsBochner}. The key point is that all weights $| S_{\alpha} \omega |^2$ are bounded by $\frac{p(n-p)}{n} | \omega |^2$ while the total weight $\sum_{\alpha} | S_{\alpha} \omega |^2 = \frac{p(n-p)}{n} \frac{n+2}{2} | \omega |^2$ is large in comparison. In particular, $\sum_{\alpha} \lambda_{\alpha} | S_{\alpha} \omega |^2 \geq 0$ if the curvature operator of the second kind is $\frac{n+2}{2}$-nonnegative. Moreover, $\frac{n+2}{2}$-nonnegativity also implies nonnegative Ricci curvature and hence $g( \Ric_L( \omega ), \omega ) \geq 0$.

It follows that every harmonic form on a manifold with $\frac{n+2}{2}$-nonnegative curvature operator of the second kind is parallel and satisfies $0 \geq \scal | \omega |^2$. This implies $\omega$ vanishes unless $M$ is flat. We remark that this final conclusion is also possible with the formula obtained by Ogiue-Tachibana \cite{OgiueTachibanaVarietesRiemCourbureRestreint}, cf. remark \ref{ConnectionToOgiueTachibana}, provided the curvature operator of the second kind is positive. However, this argument has not been pointed out before.

For the general case we also incorporate the Ricci and scalar curvature terms in a single estimate on the eigenvalues of the curvature operator of the second kind. This places different weights on the eigenvalues. The main technical tool, the weight principle \ref{WeightPrinciple}, is a refinement of the ideas above and allows us obtain eigenvalue estimates for sums with different weights. In particular, the weight principle \ref{WeightPrinciple} extends \cite[Lemma 2.1]{PetersenWinkNewCurvatureConditionsBochner} to an abstract setting.  

\vspace{2mm}

The curvature operator of the second kind also naturally occurs in the context of deformations of Einstein structures, cf. Berger-Ebin \cite{BergerEbinDecompSymTensors}, Besse \cite{BesseEinstein} or Koiso \cite{KoisoDecompSpaceOfMetrics, KoisoSecondDerivativeScal}, as well as in Bochner-Weitzenb\"ock formulas for symmetric tensors, cf. Mike\v{s}-Rovenski-Stepanov \cite{MikesExampleLichnerowiczLaplacian} or Shandra-Stepanov-Mike\v{s} \cite{ShandraHigherCodazziTensors}.

Restrictions on the restricted holonomy groups of not necessarily complete manifolds which satisfy nonnegativity or nonpositivity conditions on the eigenvalues of the curvature operator of the second kind are studied by the authors and W. Wylie in \cite{NPWWHolonomyAndCOSK}. 

\vspace{2mm}

\textit{Structure.} Section \ref{SectionPreliminaries} collects some preliminary results and sets up notation. In section \ref{SectionBochnerTechnique} we provide a brief introduction to the Bochner technique and in particular establish a Bochner formula for the curvature operator of the second kind in proposition \ref{BochnerFormulaForCurvSecondKind}. In section \ref{SectionWeightPrinciple} we prove the key technical tool, the weight principle \ref{WeightPrinciple}. As an application we obtain estimates on the curvature terms in the Bochner formula. For example, proposition \ref{WeakEigenvalueEstimate} provides a simple, preliminary estimate that provides asymptotically the same result as Theorem \ref{BettiNumbersCurvatureSecondKind}, cf. example \ref{ExplicitValuesGeneralBettinumberTheorem}. Theorem \ref{BettiNumbersCurvatureSecondKind} itself relies on the refined estimate in proposition \ref{ImprovedEigenvalueEstimate}. The proofs of the main Theorems are given in section \ref{SectionProofs}. Section \ref{SectionProofs} also contains the example of an $(n+1)$-positive algebraic curvature operator of the second kind with negative Ricci curvatures, and discusses the curvature of the rational homology sphere $SU(3)/SO(3).$

\section{Preliminaries}
\label{SectionPreliminaries}

Let $(V,g)$ be an $n$-dimensional Euclidean vector space, and let $R$ be an algebraic $(0,4)$-curvature tensor on $V.$ Let $S^2(V)$ denote the space of symmetric $(0,2)$-tensors on $V.$ The subspace of trace-free symmetric $(0,2)$-tensors is denoted by $S_0^2(V).$ Recall that 
\begin{align*}
S^2(V) = S^2_0(V) \oplus \R g
\end{align*}
is the decomposition of $S^2(V)$ into $O(n)$-invariant, irreducible subspaces. 

For an algebraic curvature tensor $R$ set 
\begin{align*}
 \overline{R} & \colon S^2(V) \to S^2(V), \\
& h \mapsto \sum_{k,l=1}^n R_{\cdot kl \cdot} h_{kl},
\end{align*}
where the components are with respect to an orthonormal basis $e_1, \ldots, e_n$ for $V.$ 

Note that $\overline{R}$ is self-adjoint and
\begin{align*}
\overline{R}(g)= \sum_{k,l=1}^n R_{\cdot kl \cdot} \delta_{kl} = - \Ric.
\end{align*}
Furthermore, the operator $\overline{R}$ leaves the subspace $S^2_0(V)$ invariant if and only if $R$ is Einstein. 

Define
\begin{align*}
 \mathring{R} \colon S^2(V) & \to S_0^2(V), \\
\mathring{R} = \pr_{S_0^2(V)} \circ \overline{R} = & \overline{R} + g(\Ric, \cdot ) \frac{g}{n}.
\end{align*}
The induced operator $\mathcal{R}=\mathring{R}_{|S_0^2(V)} \colon S_0^2(V) \to S_0^2(V)$ is called {\em curvature operator of the second kind.} Note that $\mathcal{R}$ is again self-adjoint.

\begin{example}
\normalfont
If $R$ is the curvature tensor of the round sphere, then $\overline{R}(h) = h - \tr(h) \id.$ 
Indeed, note that $R_{ijkl} = \delta_{ik} \delta_{jl} - \delta_{il} \delta_{jk}$ and thus
\begin{align*}
\left( \overline{R}(h) \right)_{ij} = \sum_{k,l}^n R_{iklj} h_{kl} 
= \sum_{k,l}^n \left( \delta_{il} \delta_{kj} - \delta_{ij} \delta_{kl} \right) h_{kl} 
= h_{ji} - \delta_{ij} \sum_{k=1}^n h_{kk} = h_{ij} - \delta_{ij} \tr(h).
\end{align*}
In particular, $\mathring{R}_{|S_0^2(V)} = \id_{S_0^2(V)}.$
\end{example}

\begin{proposition}
\label{TracesSecondCurvature}
\begin{align*}
\tr ( \overline{R} ) = \frac{\scal}{2} \ \text{ and } \ \tr ( \mathcal{R} ) = \frac{n+2}{2n} \scal.
\end{align*}
\end{proposition}
\begin{proof}
If $e_1, \ldots, e_n$ is an orthonormal basis for $V,$ then
\begin{align*}
g( \overline{R}(e^i \otimes e^i), & \  e^i \otimes e^i)  = 0, \\
g( \overline{R}( \frac{1}{\sqrt{2}} ( e^i \otimes e^j + e^j \otimes e^i )), & \ \frac{1}{\sqrt{2}}   ( e^i \otimes e^j + e^j \otimes e^i ) ) = R_{ijij}.
\end{align*}
This implies $\tr ( \overline{R} )= \sum_{i<j} R_{ijij} = \frac{\scal}{2}.$

Furthermore, note that $\tr ( \mathcal{R} ) = \tr( \overline{R} ) - g( \overline{R}( \frac{g}{\sqrt{n}} ),  \frac{g}{\sqrt{n}} ) = \frac{\scal}{2} + \frac{\scal}{n} = \frac{n+2}{2n} \scal. $
\end{proof}

Via the metric, we identify symmetric $(0,2)$-tensors with self-adjoint endomorphisms of $V.$
\begin{definition}
Let $V$ be a finite dimensional Euclidean vector space. Let $\mathcal{T}^{(0,k)}(V)$ denote the vector space of $(0,k)$-tensors on $V.$ For $S \in S^2(V)$ and $T \in \mathcal{T}^{(0,k)}(V)$ set 
\begin{align*}
(ST)(X_1, \ldots, X_k) = \sum_{i=1}^k T(X_1, \ldots, S X_i, \ldots, X_k)
\end{align*}
and define $T^{S^2} \in \mathcal{T}^{(0,k)}(V) \otimes S^2(V)$ via
\begin{align*}
g( T^{S^2}(X_1, \ldots, X_k), S) = (ST)(X_1, \ldots, X_k). 
\end{align*}
In particular, if $\lbrace S_{\alpha} \rbrace$ is an orthonormal basis for $S^2(V),$ then 
\begin{align*}
T^{S^2} = \sum_{S_{\alpha}} S_{\alpha}T \otimes S_{\alpha}.
\end{align*}

Similarly, we define 
\begin{align*}
T^{S_0^2} = \sum_{S_{\alpha}} S_{\alpha}T \otimes S_{\alpha}
\end{align*}
where $\lbrace S_{\alpha} \rbrace$ is an orthonormal basis for $S_0^2(V).$ 
\end{definition}

\begin{remark} \normalfont

(a) We have $T^{S^2} = T^{S_0^2} + \frac{1}{\sqrt{n}}(g T) \otimes \frac{1}{\sqrt{n}} g$. The observation
\begin{align*}
gT = \sum_{i=1}^k T(\ldots, \id, \ldots) = kT
\end{align*}
thus implies the important relation
\begin{equation*}
T^{S^2} = T^{S_0^2} + \frac{k}{n} T \otimes g.
\end{equation*}

(b) If $\omega \in \Ext^p V^{*}$ is a $p$-form and $S \in S^2(V)$, then $S \omega$ is again a $p$-form. Indeed,
\begin{align*}
(S \omega)(X, X, X_3, \ldots, X_p)= \omega( SX, X, X_3, \ldots, X_p) + \omega( X, SX, X_3, \ldots, X_p) + 0 = 0.
\end{align*}

(c) We use the standard norm on $\bigotimes^{k} V^{*}$. 
In particular, if $\omega$ is a $p$-form, then 
\begin{align*}
| \omega |^2 = \sum_{i_1, \ldots, i_p} (\omega_{i_1 \ldots i_p})^2 = p! \sum_{i_1 < \ldots < i_p} (\omega_{i_1 \ldots i_p})^2.
\end{align*}
\end{remark}

\begin{example}
\label{ExplicitFormulaTS2}
\normalfont
Set $e^i \bcirc e^j = \frac{1}{2} ( e^i \otimes e^j + e^j \otimes e^i)$ and note that
\begin{align*}
|| e^i \bcirc e^j ||^2 = 
\begin{cases}
1 & \text{ if } i=j, \\
\frac{1}{2} & \text{ if } i \neq j.
\end{cases}
\end{align*}
Thus,
\begin{align*}
T^{S^2} = & \ \sum_{i \leq j} (e^i \bcirc e^j)T \otimes \frac{e^i \bcirc e^j}{|| e^i \bcirc e^j||^2} \\
= & \ \sum_{i} (e^i \bcirc e^i)T \otimes e^i \bcirc e^i + \sum_{i \neq j} (e^i \bcirc e^j)T \otimes e^i \bcirc e^j \\
= & \ \sum_{i,j} (e^i \bcirc e^j)T \otimes e^i \bcirc e^j.
\end{align*}
Furthermore, note that
\begin{align*}
\overline{R}( e^i \bcirc e^j ) = \frac{1}{2} \sum_{k,l} R_{\cdot k l \cdot} \left( \delta_{ki} \delta_{lj} + \delta_{kj} \delta_{li} \right) = \frac{1}{2} \left( R_{\cdot ij \cdot } + R_{\cdot ji \cdot} \right)
\end{align*}
and thus 
\begin{align*}
g( \overline{R}( e^i \bcirc e^j ), e^k \bcirc e^l ) = \frac{1}{2} \left( R_{kijl} + R_{kjil} \right).
\end{align*}
\end{example}

\begin{definition}
For an algebraic curvature tensor $R$ and a $(0,k)$-tensor $T$ set 
\begin{align*}
\overline{R}(T^{S^2}) = \sum_{\alpha} S_{\alpha}T \otimes \overline{R}(S_{\alpha}). 
\end{align*} 
In particular, for all $S \in S^2(V)$ and $X_1, \ldots, X_k \in V$ we have 
\begin{align*}
g( \overline{R}(T^{S^2})(X_1, \ldots, X_k), S) = (\overline{R}(S)T)(X_1, \ldots, X_k).
\end{align*}

Similarly we define $\overline{R}(T^{S_0^2})$ and $\mathcal{R}(T^{S_0^2}).$
\end{definition}

\begin{example}
\label{EigenvalueFormulaForCurvatureTerm}
\normalfont 
If $\lbrace S_{\alpha} \rbrace$ is an orthonormal eigenbasis for $\mathcal{R}$ with corresponding eigenvalues $\lbrace \lambda_{\alpha} \rbrace$, then
\begin{align*}
g(\mathcal{R}(T^{S^2_0}), T^{S^2_0}) & = \sum_{\alpha, \beta} g(S_{\alpha}T, S_{\beta}T) \ g(\mathcal{R}(S_{\alpha}),S_{\beta}) \\
& = \sum_{\alpha, \beta} \lambda_{\alpha} g(S_{\alpha}T, S_{\beta}T) \ g(S_{\alpha},S_{\beta}) \\
&  = \sum_{\alpha} \lambda_{\alpha} | S_{\alpha} T |^2
\end{align*}
and thus in particular
\begin{align*}
| T^{S^2_0} |^2 = \sum_{\alpha} | S_{\alpha} T |^2.
\end{align*}
\end{example}

\begin{proposition}
\label{GeometricTermOnlyDependsOnR}
If $T$ is a $(0,k)$-tensor, then
\begin{align*}
g(\overline{R}(T^{S_0^2}), T^{S_0^2})=g(\mathcal{R}(T^{S_0^2}), T^{S_0^2}).
\end{align*}
\end{proposition}
\begin{proof}
Let $\lbrace S_{\alpha} \rbrace$ denote an orthonormal basis for $S_0^2$. Recall that $\overline{R}=\mathring{R} + g(\Ric, \cdot) \frac{g}{n}.$ Since any $S_{\alpha}$ is trace-free and hence orthogonal to $g$, we have 
\begin{align*}
g( \overline{R}(S_{\alpha}), S_{\beta} ) = g( (\pr_{S_0^2} \circ \overline{R}( S_{\alpha} ), S_{\beta} ) = g( \mathcal{R}(S_{\alpha}), S_{\beta}).
\end{align*}
Thus we obtain
\begin{align*}
g( \overline{R}(T^{S_0^2}), T^{S_0^2}) &  = \sum_{\alpha, \beta} g(S_{\alpha}T,S_{\beta}T) g( \overline{R}(S_{\alpha}), S_{\beta} ) \\
& = \sum_{\alpha, \beta} g(S_{\alpha}T,S_{\beta}T) g( \mathcal{R}(S_{\alpha}), S_{\beta} ) 
= g( \mathcal{R}(T^{S_0^2}), T^{S_0^2}).
\end{align*}
\end{proof}

\section{A Bochner formula for the curvature operator of the second kind}
\label{SectionBochnerTechnique}

The Bochner technique relies on the observation that on $p$-forms
\begin{align*}
\Delta_{\text{Hodge}} = \nabla^{*} \nabla + \Ric_L,
\end{align*}
where 
\begin{align*}
\Ric_L (\omega)(X_1, \ldots, X_p) = \sum_{i=1}^p \sum_{j=1}^n (R(X_i, e_j) \omega)(X_1, \ldots, e_j, \ldots, X_p).
\end{align*}
In particular, cf. \cite{BesseEinstein}, the curvature term on $p$-forms is given by
\begin{align*}
g(\Ric_L(\omega), \omega) = & \ p \sum_{i_2, \ldots, i_p} \sum_{i,j} R_{ij} \omega_{i i_2 \ldots i_p} \omega_{j i_2 \ldots i_p} \\
& \ - \frac{p (p-1)}{2} \sum_{i_3, \ldots, i_p} \sum_{i,j,k,l} R_{ijkl} \omega_{i j i_3 \ldots i_p} \omega_{k l i_3 \ldots i_p}.
\end{align*}

If $\omega$ is a harmonic $p$-form, $\Delta_{\text{Hodge}} \omega = 0,$ then
\begin{align*}
\Delta \frac{1}{2} | \omega |^2 = | \nabla \omega |^2 - g( \nabla^{*} \nabla \omega, \omega ) = | \nabla \omega |^2 + g( \Ric_L( \omega ), \omega ).
\end{align*}
In particular, if $M$ is compact and $g( \Ric_L( \omega ), \omega ) \geq 0,$ then $\omega$ is parallel. 

Estimation results on the dimension of the kernel of the Hodge Laplacian follow if there are constants $\kappa \leq 0$ and $C>0$ such that $g( \Ric_L( \omega ), \omega ) \geq \kappa C | \omega |^2$ and $\Ric \geq (n-1) \kappa,$ cf. \cite[Theorem 1.9]{PetersenWinkNewCurvatureConditionsBochner}. \vspace{2mm}

The connection to the curvature operator of the second kind is given by the following observation.

\begin{proposition}
\label{BochnerFormulaForCurvSecondKind}
Let $R$ be an algebraic curvature tensor and let $\omega$ be a $p$-form. With respect to an orthonormal basis the curvature term in the Bochner formula satisfies
\begin{align*}
\frac{3}{2} g( \Ric_L(\omega), \omega) = g( \mathcal{R}(\omega^{S_0^2}), \omega^{S_0^2}) + \frac{p(n-2p)}{n} \sum_{j,k} \sum_{i_2, \ldots, i_p} R_{jk} \omega_{j i_2 \ldots i_p} \omega_{k i_2 \ldots i_p} + \frac{p^2}{n^2} \scal | \omega |^2.
\end{align*}
If the orthonormal basis diagonalizes the Ricci tensor, then furthermore
\begin{align*}
\frac{3}{2} g( \Ric_L(\omega), \omega) = g( \mathcal{R}(\omega^{S_0^2}), \omega^{S_0^2}) + \frac{n-2p}{n} \sum_{I=(i_1, \ldots, i_p)} \left( \sum_{i \in I} R_{ii} \right) \omega_{I}^2 + \frac{p^2}{n^2} \scal | \omega |^2.
\end{align*}
In particular, if $R$ is Einstein, $\Ric = \frac{\scal}{n}g,$ then
\begin{align*}
\frac{3}{2} g( \Ric_L(\omega), \omega) = g( \mathcal{R}(\omega^{S_0^2}), \omega^{S_0^2}) + \frac{p(n-p)}{n^2} \scal | \omega |^2.
\end{align*}
\end{proposition}
\begin{proof}
Recall that due to proposition \ref{GeometricTermOnlyDependsOnR} we may consider the term $g(\overline{R}(\omega^{S_0^2}),\omega^{S_0^2}).$ Since $\omega^{S_0^2} = \omega^{S^2} - \frac{p}{n} \omega \otimes g$ and $\overline{R}(g)=- \Ric,$ it follows that 
\begin{align*}
g(\overline{R}(\omega^{S_0^2}),\omega^{S_0^2}) = g(\overline{R}(\omega^{S^2}),\omega^{S^2}) + \frac{2p}{n} g(\omega^{S^2}, \omega \otimes \Ric) - \frac{p^2}{n^2} \scal |\omega|^2.
\end{align*}
Propositions \ref{InnerProductOmegaS2Ric} and \ref{CurvTermOnOmegaSym} below imply 
\begin{align*}
g( \overline{R} ( \omega^{S_0^2} ), \omega^{S_0^2} ) 
= & \ - \frac{3}{2} \frac{p(p-1)}{2} \sum_{i,j,k,l} \sum_{i_3, \ldots, i_p } \omega_{i j i_3 \ldots i_p} \omega_{kl i_3 \ldots i_p } R_{ijkl} \\
& \ + \left( \frac{p}{2} + \frac{2p^2}{n} \right) \sum_{i,j} \sum_{i_2, \ldots, i_p } R_{ij} \omega_{i i_2 \ldots i_p} \omega_{j i_2 \ldots i_p}  - \frac{p^2}{n^2} \scal | \omega |^2 \\
= & \ - \frac{3}{2} \frac{p(p-1)}{2} \sum_{i,j,k,l} \sum_{i_3, \ldots, i_p } \omega_{i j i_3 \ldots i_p} \omega_{kl i_3 \ldots i_p } R_{ijkl} \\
& \ + \frac{p(n+4p)}{2n} \sum_{i,j} \sum_{i_2, \ldots, i_p } R_{ij} \omega_{i i_2 \ldots i_p} \omega_{j i_2 \ldots i_p} - \frac{p^2}{n^2} \scal | \omega |^2.
\end{align*}
Thus,
\begin{align*}
g( \mathcal{R} & (\omega^{S_0^2}),  \omega^{S_0^2}) + \frac{p(n-2p)}{n} \sum_{j,k} \sum_{i_2, \ldots, i_p} R_{jk} \omega_{j i_2 \ldots i_p} \omega_{k i_2 \ldots i_p} + \frac{p^2}{n^2} \scal | \omega |^2 = \\
& = \ \frac{3p}{2} \sum_{i_2, \ldots, i_p} \sum_{i,j} R_{ij} \omega_{i i_2 \ldots i_p} \omega_{j i_2 \ldots i_p} - \frac{3}{2} \frac{p (p-1)}{2} \sum_{i_3, \ldots, i_p} \sum_{i,j,k,l} R_{ijkl} \omega_{i j i_3 \ldots i_p} \omega_{k l i_3 \ldots i_p} \\
& = \ \frac{3}{2} g( \Ric_L(\omega), \omega ).
\end{align*}
\end{proof}

\begin{proposition}
\label{InnerProductsEiEjOmega}
For an orthonormal basis $e_1, \ldots, e_n$ set $e^i \bcirc e^j = \frac{1}{2} \left( e^i \otimes e^j + e^j \otimes e^i \right).$  Then, every $p$-form $\omega$ satisfies
\begin{align*}
g( (e^i \bcirc e^j) \omega, \omega ) = p \sum_{i_2, \ldots, i_p} \omega_{i i_2 \ldots i_p} \omega_{j i_2 \ldots i_p}
\end{align*}
and
\begin{align*}
g( (e^i \bcirc e^j) \omega, (e^k \bcirc e^l) \omega ) = & \ \frac{p}{4} \sum_{i_2, \ldots, i_p} \left( \delta_{jk} \omega_{i i_2 \ldots i_p}\omega_{l i_2 \ldots i_p} + \delta_{jl} \omega_{i i_2 \ldots i_p}\omega_{k i_2 \ldots i_p} \right. \\
& \hspace{17mm} \left. + \delta_{ik} \omega_{j i_2 \ldots i_p}\omega_{l i_2 \ldots i_p} + \delta_{il} \omega_{j i_2 \ldots i_p}\omega_{k i_2 \ldots i_p} \right) \\
& \ + \frac{p(p-1)}{2} \sum_{i_3, \ldots, i_p} \left( \omega_{ik i_3 \ldots i_p} \omega_{jl i_3 \ldots i_p} + \omega_{il i_3 \ldots i_p} \omega_{jk i_3 \ldots i_p} \right).
\end{align*}
\end{proposition}
\begin{proof}
If $\Omega$ is another $p$-form, then
\begin{align*}
g( (e^i \bcirc e^j) \omega, \Omega ) = & \ \sum_{i_1, \ldots, i_p} ((e^i \bcirc e^j) \omega)_{i_1 \ldots i_p} \Omega_{i_1 \ldots i_p} \\
= & \ \frac{1}{2} \sum_{i_1, \ldots, i_p} \sum_{k=1}^p \delta_{i i_k} \omega_{i_1 \ldots j \ldots i_p} \Omega_{i_1 \ldots i_p} + \frac{1}{2} \sum_{i_1, \ldots, i_p} \sum_{k=1}^p \delta_{j i_k} \omega_{i_1 \ldots i \ldots i_p} \Omega_{i_1 \ldots i_p}  \\
= & \ \frac{1}{2} \sum_{k=1}^p \ \sum_{i_1, \ldots, i_{k-1}, i_{k+1}, \ldots, i_p} \ \left( \omega_{i_1 \ldots j \ldots i_p} \Omega_{i_1 \ldots i \ldots i_p} +\omega_{i_1 \ldots i \ldots i_p} \Omega_{i_1 \ldots j \ldots i_p} \right) \\
= & \ \frac{p}{2} \sum_{i_2, \ldots, i_p} \left( \omega_{i i_2 \ldots i_p} \Omega_{j i_2 \ldots i_p} + \omega_{j i_2 \ldots i_p} \Omega_{i i_2 \ldots i_p} \right).
\end{align*}
This implies the first claim. Moreover, note that
\begin{align*}
\left( (e^k \bcirc e^l) \omega \right)_{i i_2 \ldots i_p} = \frac{1}{2} \left( \delta_{ik} \omega_{l i_2 \ldots i_p} + \delta_{il} \omega_{k i_2 \ldots i_p} \right) + \frac{1}{2} \sum_{\alpha=2}^p \left( \delta_{k i_{\alpha}} \omega_{i i_2 \ldots l \ldots i_p} + \delta_{l i_{\alpha}} \omega_{i i_2 \ldots k \ldots i_p} \right).
\end{align*}
Thus,
\begin{align*}
g( (e^i \bcirc e^j) \omega, (e^k \bcirc e^l) \omega ) = & \ \frac{p}{4} \sum_{i_2, \ldots, i_p} \left( \delta_{jk} \omega_{i i_2 \ldots i_p} \omega_{l i_2 \ldots i_p} + \delta_{jl} \omega_{i i_2 \ldots i_p} \omega_{k i_2 \ldots i_p} \right) \\ 
& \ + \frac{p}{4} \sum_{i_2, \ldots, i_p} \sum_{\alpha=2}^p \left( \delta_{k i_{\alpha}} \omega_{i i_2 \ldots i_p} \omega_{j i_2 \ldots l \ldots i_p} + \delta_{l i_{\alpha}} \omega_{i i_2 \ldots i_p} \omega_{j i_2 \ldots k \ldots i_p} \right) \\
& \ + \text{both sums with } i \text{ and } j \text{ reversed} \\
= & \ \frac{p}{4} \sum_{i_2, \ldots, i_p} \left( \delta_{jk} \omega_{i i_2 \ldots i_p} \omega_{l i_2 \ldots i_p} + \delta_{jl} \omega_{i i_2 \ldots i_p} \omega_{k i_2 \ldots i_p} \right. \\
& \hspace{17mm} \left.  + \delta_{ik} \omega_{j i_2 \ldots i_p} \omega_{l i_2 \ldots i_p} + \delta_{il} \omega_{j i_2 \ldots i_p} \omega_{k i_2 \ldots i_p} \right) \\
& \ + \frac{p}{2} \ \sum_{\alpha=2}^p \sum_{i_2, \ldots, i_{\alpha-1}, i_{\alpha+1}, \ldots, i_p} \left( \omega_{i i_2 \ldots k \ldots i_p} \omega_{j i_2 \ldots l \ldots i_p} + \omega_{i i_2 \ldots l \ldots i_p} \omega_{j i_2 \ldots k \ldots i_p} \right) \\
= & \ \frac{p}{4} \sum_{i_2, \ldots, i_p} \left( \delta_{jk} \omega_{i i_2 \ldots i_p} \omega_{l i_2 \ldots i_p} + \delta_{jl} \omega_{i i_2 \ldots i_p} \omega_{k i_2 \ldots i_p} \right. \\
& \hspace{17mm} \left.  + \delta_{ik} \omega_{j i_2 \ldots i_p} \omega_{l i_2 \ldots i_p} + \delta_{il} \omega_{j i_2 \ldots i_p} \omega_{k i_2 \ldots i_p} \right) \\
& \ + \frac{p(p-1)}{2} \ \sum_{i_3, \ldots, i_p} \left( \omega_{i k i_3 \ldots i_p} \omega_{j l i_3 \ldots i_p} + \omega_{i li_3 \ldots i_p} \omega_{j k i_3 \ldots i_p} \right).
\end{align*}
\end{proof}

\begin{proposition}
\label{InnerProductOmegaS2Ric}
Every $p$-form $\omega$ satisfies
\begin{align*}
g( \omega^{S^2}, \omega \otimes \Ric ) = p \sum_{i,j} \sum_{i_2, \ldots, i_p} R_{ij} \omega_{i i_2 \ldots i_p} \omega_{j i_2 \ldots i_p}.
\end{align*}
\end{proposition}
\begin{proof}
Due to example \ref{ExplicitFormulaTS2} we have
\begin{align*}
g( \omega^{S^2}, \omega \otimes \Ric ) = \sum_{i,j} g( (e^i \bcirc e^j) \omega \otimes e^i \bcirc e^j, \omega \otimes \Ric ) = \sum_{i,j} g( (e^i \bcirc e^j) \omega, \omega) R_{ij} 
\end{align*}
and thus proposition \ref{InnerProductsEiEjOmega} implies the claim.
\end{proof}

\begin{proposition}
\label{CurvTermOnOmegaSym}
If $\omega$ is a $p$-form, then
\begin{align*}
g( \overline{R}( \omega^{S^2} ), \omega^{S^2} ) = & \frac{p}{2} \sum_{i,j} \sum_{i_2, \ldots, i_p} R_{ij} \omega_{i i_2 \ldots i_p} \omega_{j i_2 \ldots i_p} - \frac{3}{2} \frac{p(p-1)}{2} \sum_{i,j,k,l} \sum_{i_3, \ldots, i_p} R_{ijkl} \omega_{ij i_3 \ldots i_p} \omega_{kl i_3 \ldots i_p}
\end{align*}
and thus
\begin{align*}
\frac{3}{2} g(\Ric_L( \omega ), \omega ) = g( \overline{R}( \omega^{S^2} ), \omega^{S^2} ) + p \sum_{i,j} \sum_{i_2, \ldots, i_p} R_{ij} \omega_{i i_2 \ldots i_p} \omega_{j i_2 \ldots i_p}.
\end{align*}
\end{proposition}
\begin{proof}
Notice that
\begin{align*}
\sum_{i,j,k,l} \sum_{i_3, \ldots, i_p} \omega_{ik i_3 \ldots i_p}  \omega_{jl i_3 \ldots i_p} R_{ijkl} = \frac{1}{2} \sum_{i,j,k,l} \sum_{i_3, \ldots, i_p}  \omega_{ij i_3 \ldots i_p}  \omega_{kl i_3 \ldots i_p} R_{ijkl},
\end{align*}
since $R_{ijkl} = -  \left( R_{jkil} + R_{kijl} \right)=    R_{ilkj} + R_{ikjl} $ implies
\begin{align*}
\sum_{i,j,k,l} \sum_{i_3, \ldots, i_p} \omega_{ik i_3 \ldots i_p}  \omega_{jl i_3 \ldots i_p} R_{ijkl} 
= & \ \sum_{i,j,k,l} \sum_{i_3, \ldots, i_p} \left( \omega_{ik i_3 \ldots i_p}  \omega_{jl i_3 \ldots i_p} R_{ilkj} + \omega_{ik i_3 \ldots i_p}  \omega_{jl i_3 \ldots i_p} R_{ikjl} \right) \\
= & \ - \sum_{i,j,k,l} \sum_{i_3, \ldots, i_p} \omega_{ik i_3 \ldots i_p}  \omega_{jl i_3 \ldots i_p} R_{ijkl} \\
& \ + \sum_{i,j,k,l} \sum_{i_3, \ldots, i_p}  \omega_{ij i_3 \ldots i_p}  \omega_{kl i_3 \ldots i_p} R_{ijkl}.
\end{align*}
Due to example \ref{ExplicitFormulaTS2} and proposition \ref{InnerProductsEiEjOmega} we thus obtain
\begin{align*}
g( \overline{R}( \omega^{S^2}), \omega^{S^2} ) = & \ \sum_{i,j,k,l} g( (e^i \bcirc e^j) \omega \otimes \overline{R} (e^i \bcirc e^j), (e^k \bcirc e^l) \omega \otimes \overline{R} (e^k \bcirc e^l) ) \\
= & \ \frac{1}{4}\sum_{i,j,k,l} g( (e^i \bcirc e^j) \omega,  (e^k \bcirc e^l) \omega ) ( R_{kijl} + R_{kjil} + R_{lijk} + R_{lijk} ) \\
= & \ \frac{1}{2} \sum_{i,j,k,l} g( (e^i \bcirc e^j) \omega,  (e^k \bcirc e^l) \omega ) ( R_{kijl} + R_{kjil} ) \\
= & \ \frac{p}{8} \sum_{i,j,k,l} \sum_{i_2, \ldots, i_p} \left( \delta_{jk} \omega_{i i_2 \ldots i_p}\omega_{l i_2 \ldots i_p} + \delta_{jl} \omega_{i i_2 \ldots i_p}\omega_{k i_2 \ldots i_p} \right. \\
& \hspace{25mm} \left. + \delta_{ik} \omega_{j i_2 \ldots i_p}\omega_{l i_2 \ldots i_p} + \delta_{il} \omega_{j i_2 \ldots i_p}\omega_{k i_2 \ldots i_p} \right) \left( R_{kijl} + R_{kjil} \right) \\
& \ + \frac{p(p-1)}{4} \sum_{i,j,k,l} \sum_{i_3, \ldots, i_p} \left( \omega_{ik i_3 \ldots i_p} \omega_{jl i_3 \ldots i_p} + \omega_{il i_3 \ldots i_p} \omega_{jk i_3 \ldots i_p} \right) \left(  R_{kijl} + R_{kjil} \right) \\
= & \ \frac{p}{8} \sum_{i_2, \ldots, i_p} \left( \sum_{i,l} R_{il} \omega_{i i_2 \ldots i_p}\omega_{l i_2 \ldots i_p} + \sum_{i,k} R_{ik} \omega_{i i_2 \ldots i_p}\omega_{k i_2 \ldots i_p} \right. \\
& \hspace{20mm} \left. + \sum_{j,l} R_{jl} \omega_{j i_2 \ldots i_p} \omega_{l i_2 \ldots i_p} + \sum_{j,k} R_{kj} \omega_{j i_2 \ldots i_p}\omega_{k i_2 \ldots i_p} \right) \\
& \ + \frac{p(p-1)}{4} \sum_{i,j,k,l} \sum_{i_3, \ldots, i_p} \left( \omega_{ji i_3 \ldots i_p} \omega_{kl i_3 \ldots i_p} + \omega_{ki i_3 \ldots i_p} \omega_{jl i_3 \ldots i_p} \right. \\
& \hspace{42mm} \left. + \omega_{jl i_3 \ldots i_p} \omega_{ki i_3 \ldots i_p} + \omega_{kl i_3 \ldots i_p} \omega_{ji i_3 \ldots i_p} \right) R_{ijkl} \\
= & \ \frac{p}{2} \sum_{i,j} \sum_{i_2, \ldots, i_p} R_{ij} \omega_{i i_2 \ldots i_p} \omega_{j i_2 \ldots i_p} \\
& \ - \frac{p(p-1)}{2} \sum_{i,j,k,l} \sum_{i_3, \ldots, i_p} \left( \omega_{ij i_3 \ldots i_p} \omega_{kl i_3 \ldots i_p} + \omega_{ik i_3 \ldots i_p} \omega_{jl i_3 \ldots i_p} \right) R_{ijkl} \\
= & \ \frac{p}{2} \sum_{i,j} \sum_{i_2, \ldots, i_p} R_{ij} \omega_{i i_2 \ldots i_p} \omega_{j i_2 \ldots i_p} 
- \frac{3}{2} \frac{p(p-1)}{2} \sum_{i,j,k,l} \sum_{i_3, \ldots, i_p} \omega_{ij i_3 \ldots i_p} \omega_{kl i_3 \ldots i_p}  R_{ijkl},
\end{align*}
where we used the initial observation for the last equality.
\end{proof}

\begin{remark}
\normalfont
Every $(0,p)$-tensor $T$ satisfies
\begin{align*}
\frac{3}{2} g( \Ric_L(T),T) = & \ g( \mathcal{R}(T^{S_0^2}), T^{S_0^2}) + \frac{p(n-2p)}{n} \sum_{j,k} \sum_{i_2, \ldots, i_p} R_{jk} T_{j i_2 \ldots i_p} T_{k i_2 \ldots i_p} + \frac{p^2}{n^2} \scal | T |^2 \\
& \ + \sum_{1 \leq r \neq s \leq p} \  \sum_{I \in \mathcal{I}^{rs}} \ \sum_{i,j,k,l} T_{I_{ij}^{rs}} T_{I_{kl}^{rs}}  \left( R_{kijl} + R_{kjil} \right),
\end{align*}
where
\begin{align*}
\mathcal{I}^{rs} & = \lbrace (i_1, \ldots, i_{r-1}, i_{r+1}, \ldots i_{s-1}, i_{s+1}, \ldots, i_p) \in \lbrace 1, \ldots, n \rbrace^{p-2} \rbrace, \\
T_{I_{ij}^{rs}} & = T_{i_1 \ldots i_{r-1} i i_{r+1} \ldots i_{s-1} j i_{s+1} \ldots i_p}.
\end{align*}

Note that the last term vanishes for $p$-forms. For symmetric $(0,2)$-tensors the last term reads
\begin{align*}
4 \sum_{i,j,k,l} T_{ij} T_{kl} R_{kijl} = 8 \sum_{i<j} \sum_{k<l} T_{ij} T_{kl} \left( R_{kijl} + R_{lijk} \right). 
\end{align*}
\end{remark}

\begin{remark}
\label{ConnectionToOgiueTachibana}
\normalfont 
In order to recover the Bochner formula of Ogiue-Tachibana \cite{OgiueTachibanaVarietesRiemCourbureRestreint} from proposition \ref{BochnerFormulaForCurvSecondKind}, set 
\begin{align*}
e^i \odot e^j = e^i \otimes e^j + e^j \otimes e^i - \frac{2}{n} \delta_{ij} \sum_{k=1}^n e^k \otimes e^k.
\end{align*}
Note that the $e^i \odot e^j$ are trace-free but {\em not} orthogonal. For every $p$-form $\omega$ we have
\begin{align*}
((e^i \odot e^j) \omega)_{i_1 \ldots i_p} = \sum_{k=1}^p \left( \delta_{i i_k} \omega_{i_1 \ldots j \ldots i_p} + \delta_{j i_k} \omega_{i_1 \ldots i \ldots i_k} \right) - \frac{2}{n} \delta_{ij} \omega_{i_1 \ldots i_p}.
\end{align*}
In \cite{OgiueTachibanaVarietesRiemCourbureRestreint}, Ogiue-Tachibana observed that 
\begin{align*}
\frac{3}{2} g( \Ric_L(\omega), \omega) = & \ \frac{1}{4} \sum_{i,j,k,l} \sum_{i_1, \ldots, i_p} ((e^i \odot e^l) \omega)_{i_1 \ldots i_p} ((e^j \odot e^k) \omega)_{i_1 \ldots i_p} R_{ijkl} \\
& \ + \frac{p(n-2p)}{n} \sum_{j,k} \sum_{i_2, \ldots, i_p} R_{jk} \omega_{j i_2 \ldots i_p} \omega_{k i_2 \ldots i_p} + \frac{p^2}{n^2} \scal | \omega |^2.
\end{align*}
In fact, it is straightforward to check that for any $(0,k)$-tensor 
\begin{align*}
T^{S_0^2} = T^{S^2} - \frac{k}{n} T \otimes g = \frac{1}{2} \sum_{i,j} (e^i \odot e^j)T \otimes (e^i \bcirc e^j) = \frac{1}{4} \sum_{i,j} (e^i \odot e^j)T \otimes (e^i \odot e^j).
\end{align*}
With $g( \overline{R}(e^i \bcirc e^j), e^k \bcirc e^l) = \frac{1}{2} \left( R_{kijl} + R_{kjil} \right)$ we thus directly obtain
\begin{align*}
g( \overline{R}( T^{S_0^2} ), T^{S_0^2}) = & \ \frac{1}{4} \sum_{i,j,k,l} \sum_{i_1, \ldots, i_p} g( (e^i \odot e^j)T, (e^k \odot e^l)T ) g( \overline{R}(e^i \bcirc e^j), e^k \bcirc e^l) \\
= & \ \frac{1}{4} \sum_{i,j,k,l} \sum_{i_1, \ldots, i_p} ((e^i \odot e^l) T)_{i_1 \ldots i_p} ((e^j \odot e^k) T)_{i_1 \ldots i_p} R_{ijkl}
\end{align*}
and together with proposition \ref{BochnerFormulaForCurvSecondKind} we recover the formula of Ogiue-Tachibana.

\end{remark}

\section{The weight principle}
\label{SectionWeightPrinciple}

Let $\mathcal{R}$ be an operator with eigenvalues $\lambda_i \in \R.$ In this section we introduce a calculus to estimate finite weighted sums $\sum_i \omega_i \lambda_i$ with weights $\omega_i \geq 0.$ The main result is the weight principle \ref{WeightPrinciple}. As an application, we estimate the curvature term in the Bochner formula for the curvature operator of the second kind.

\begin{definition}
\normalfont
Let $\omega_i \geq 0$ with $\Omega = \max_i \omega_i$ and set $\mathcal{S} = \sum_i \omega_i.$ We call $\mathcal{S}$ the {\em total weight} and $\Omega$ the {\em highest weight}. 

We will use the notation
\begin{align*}
[\mathcal{R}, \Omega, \mathcal{S}]
\end{align*}
to denote any finite weighted sum $\sum_i \omega_i \lambda_i$ in terms of the eigenvalues $\lambda_i$ of the operator $\mathcal{R}$ with highest weight $\Omega$ and total weight $\mathcal{S}$. In particular, if $F(R)$ is a (geometric) quantity depending on $R$ and $\mathcal{R}=\mathcal{R}(R)$ is an operator, then we will write
\begin{align*}
F(R) \geq [\mathcal{R}, \Omega, \mathcal{S}]
\end{align*}
provided $F(R)$ is bounded from below by a weighted sum in terms of the eigenvalues of $\mathcal{R}$ with highest weight $\Omega$ and total weight $\mathcal{S}.$
\end{definition}

We write 
\begin{align*}
[\mathcal{R}, \Omega, \mathcal{S}] \geq [\mathcal{R}, \widetilde{\Omega}, \widetilde{\mathcal{S}}]
\end{align*}
provided for every sum $\sum_i \omega_i \lambda_i$ with $\sum_i \omega_i = \mathcal{S}$ and $\max_i \omega_i = \Omega$ there is 
a sum $\sum_i \tilde{\omega}_i \lambda_i$ with $\sum_i \tilde{\omega}_i = \widetilde{\mathcal{S}}$ and $\max_i \tilde{\omega}_i = \widetilde{\Omega}$ such that
\begin{align*}
\sum_i \omega_i \lambda_i \geq \sum_i \tilde{\omega}_i\lambda_i.
\end{align*}

Similarly, if $c \in \R,$ we write
\begin{align*}
[\mathcal{R}, \Omega, \mathcal{S}] \geq c
\end{align*}
provided every sum $\sum_i \omega_i \lambda_i$ with $\sum_i \omega_i = \mathcal{S}$ and $\max_i \omega_i = \Omega$ satisfies
\begin{align*}
\sum_i \omega_i \lambda_i \geq c.
\end{align*}

\begin{example}
\normalfont
If $\mathcal{R}$ denotes the curvature operator of the second kind of an $n$-dimensional Riemannian manifold, then 
\begin{align*}
\scal \geq \frac{2n}{n+2} \left[ \mathcal{R}, 1, \frac{(n-1)(n+2)}{2} \right],
\end{align*}
since $\scal = \frac{2n}{n+2} \tr( \mathcal{R} )$ and $\dim S_0^2(TM)= \frac{(n-1)(n+2)}{2}.$
\end{example}

\begin{lemma}
\label{ComputationsWithWeights}
Let $[\mathcal{R}, \Omega, \mathcal{S}], [\mathcal{R}, \widetilde{\Omega}, \widetilde{\mathcal{S}}]$ denote weighted sums of eigenvalues of $\mathcal{R}$ with highest weights $\Omega, \widetilde{\Omega}$ and total weights $\mathcal{S}, \widetilde{\mathcal{S}}$, respectively.
\begin{enumerate}
\item If $c>0$, then 
\begin{align*}
[\mathcal{R}, c \Omega, c \mathcal{S}] = c \cdot [\mathcal{R}, \Omega, \mathcal{S}].
\end{align*}
\item If $\Omega \leq \widetilde{\Omega},$ then 
\begin{align*}
[ \mathcal{R} , \Omega, \mathcal{S}] \geq [ \mathcal{R} , \widetilde{\Omega}, \mathcal{S}].
\end{align*}
\item \begin{align*}
[\mathcal{R}, \Omega, \mathcal{S}] + [\mathcal{R}, \widetilde{\Omega}, \widetilde{\mathcal{S}}] \geq [\mathcal{R}, \Omega + \widetilde{\Omega}, \mathcal{S} + \widetilde{\mathcal{S}}].
\end{align*}
\end{enumerate}
\end{lemma}
\begin{proof}
Part (a) is immediate. For part (b) note that any sum $[\mathcal{R}, \Omega, \mathcal{S}] = \sum_i \omega_i \lambda_i$ is bounded from below by the corresponding sum with decreasing weights $\omega_j \geq \omega_{j+1}$ and increasing $\lambda_j \leq \lambda_{j+1}.$ Increasing the highest weight in the rearranged sum while keeping the total weight fixed decreases the total sum.  
For part (c) note that the highest weight is bounded by $\Omega + \widetilde{\Omega}$ and its total weight is $\mathcal{S} + \widetilde{\mathcal{S}}$. Thus the claim follows from (b).
\end{proof}

\begin{lemma}
\label{MinimalSum}
If $\lambda_1 \leq \ldots \leq \lambda_N$ denote the eigenvalues of $\mathcal{R}$, then for $m \in \N$
\begin{align*}
[ \mathcal{R}, \Omega, \mathcal{S} ] \geq \left( \mathcal{S} - m \Omega \right) \lambda_{m+1} + \Omega \sum_{i=1}^m \lambda_i.
\end{align*}
\end{lemma}
\begin{proof}
If $\omega_i$ denote the corresponding weights with $\Omega = \max \omega_i$ and $\mathcal{S} = \sum_i \omega_i,$ then

\begin{align*}
[ \mathcal{R}, \Omega, \mathcal{S} ] = & \ \sum_{i=1}^N \omega_i \lambda_i \geq  \sum_{i=1}^m \omega_i \lambda_i + \sum_{i=m+1}^{N} \omega_i \lambda_{m+1} 
=  \mathcal{S} \lambda_{m+1} + \sum_{i=1}^m \omega_i \left( \lambda_i - \lambda_{m+1} \right) \\
\geq & \ \mathcal{S} \lambda_{m+1} + \Omega \sum_{i=1}^m \left( \lambda_i - \lambda_{m+1} \right) 
=  \left( \mathcal{S} - m \Omega \right) \lambda_{m+1} + \Omega \sum_{i=1}^m \lambda_i.
\end{align*}
\end{proof}

Recall that by definition $\mathcal{R}$ is $k$-nonnegative for some $k \geq 1$ provided its eigenvalues $\lambda_1 \leq \lambda_2 \leq \ldots \leq \lambda_N$ satisfy $\lambda_1 + \ldots + \lambda_{\floor{k}} + \left( k - \floor{k} \right) \lambda_{\floor{k}+1} \geq 0$.

\begin{proposition} 
\label{EquivalenceKnonneg}
Let $\lambda_1 \leq \lambda_2 \leq \ldots \leq \lambda_N$ denote the eigenvalues of $\mathcal{R}.$
\begin{enumerate}
\item $\mathcal{R}$ is $k$-nonnegative if and only if $[\mathcal{R}, 1, k] \geq 0.$
\item Let $c \in \R.$ Then, $\lambda_1 + \ldots + \lambda_{\floor{k}} + \left( k - \floor{k} \right) \lambda_{\floor{k}+1} \geq c$ if and only if $[\mathcal{R}, 1, k ] \geq c.$
\end{enumerate}
\end{proposition}
\begin{proof}
(a) By definition we have $\lambda_1 + \ldots + \lambda_{\floor{k}} + \left( k - \floor{k} \right) \lambda_{\floor{k}+1} \geq [\mathcal{R}, 1, k]$. On the other hand, by lemma \ref{MinimalSum}, any sum in $[\mathcal{R}, 1, k]$ is bounded from below by $\lambda_1 + \ldots + \lambda_{\floor{k}} + \left( k - \floor{k} \right) \lambda_{\floor{k}+1}$.

(b) follows as in (a).
\end{proof}

\begin{theorem}[Weight principle]
\label{WeightPrinciple}
Let $\mathcal{R}$ be an operator on a finite dimensional vector space with real eigenvalues. Then,
\begin{enumerate}
\item $[ \mathcal{R} , \Omega, \mathcal{S}] > 0$ if and only if $[ \mathcal{R} , 1, \frac{\mathcal{S}}{\Omega} ] > 0$ if and only if $\mathcal{R}$ is $\frac{\mathcal{S}}{\Omega}$-positive. 
\item $[ \mathcal{R} , \Omega, \mathcal{S}] \geq 0$ if and only if $[ \mathcal{R} , 1, \frac{\mathcal{S}}{\Omega} ] \geq 0$ if and only if $\mathcal{R}$ is $\frac{\mathcal{S}}{\Omega}$-nonnegative. 
\item Let $\kappa \in \R $. $[ \mathcal{R}, \Omega, \mathcal{S}] \geq \mathcal{S} \kappa$ if and only if $[ \mathcal{R} , 1, \frac{\mathcal{S}}{\Omega} ] \geq \kappa \frac{\mathcal{S}}{\Omega}.$
\item Let $k' < k$. If $\mathcal{R}$ is $k'$-nonnegative, then either $\mathcal{R}$ is $k$-positive or $1$-nonnegative.
\end{enumerate}
\end{theorem}
\begin{proof}
Parts (a)-(c) are an immediate consequence of lemma \ref{ComputationsWithWeights} and proposition \ref{EquivalenceKnonneg}. For part (d) observe that if $\mathcal{R}$ is not $k$-positive, then $\lambda_{\floor{k'}+1} = 0$. Thus, $k'$-nonnegativity implies that $\lambda_1 = \ldots = \lambda_{\floor{k'}+1} = 0$ and in particular $\lambda_i \geq 0$ for all $i.$
\end{proof}

\begin{lemma}
\label{WeightsSecondCurvatureOnForms}
Let $\omega$ be a $p$-form and $S \in S_0^2(V).$ Then,
\begin{enumerate}
\item 
\begin{align*}
| \omega^{S_0^2} |^2 = \frac{p(n-p)}{n} \frac{(n+2)}{2} | \omega |^2,
\end{align*}
\item 
\begin{align*}
| S \omega |^2 \leq \frac{p(n-p)}{n} |S|^2 | \omega|^2 = \frac{2}{n+2} |S|^2 | \omega^{S_0^2} |^2.
\end{align*}
\end{enumerate}
\end{lemma}
\begin{proof}
(a) We apply the formula for  $g(\overline{R}( \omega^{S_0^2}), \omega^{S_0^2})$ in the proof of proposition \ref{BochnerFormulaForCurvSecondKind} to the curvature tensor $R_{ijkl}=\delta_{ik} \delta_{jl} - \delta_{il} \delta_{jk}$ of the round sphere. Hence we have $\overline{R}= \id$ on $S^2_0(V)$ and 
\begin{align*}
| \omega^{S_0^2} |^2 = \frac{p}{2n} \left( -3(p-1)n + (n+4p)(n-1) - 2 p (n-1) \right) | \omega |^2 = \frac{p}{2n} (n+2)(n-p) | \omega |^2.
\end{align*}
(b) For $S \in S_0^2(V)$ there is an orthonormal basis $e_1, \ldots, e_n$ for $V$ and $\lambda_1, \ldots, \lambda_n \in \R$ such that $S(e_i) = \mu_i e_i$ for $i=1, \ldots, n.$ It follows that
\begin{align*}
(S \omega)_{i_1 \ldots i_p} = \left( \sum_{i \in \lbrace i_1, \ldots, i_p \rbrace} \mu_i \right) \omega_{i_1 \ldots i_p}
\end{align*}
and 
\begin{align*}
| S \omega |^2 = \sum_{I=(i_1, \ldots, i_p)}  \left( \sum_{i \in I} \mu_i \right)^2 \left( \omega_{i_1 \ldots i_p} \right)^2.
\end{align*}
Maximizing $|S \omega|^2$ under the constraints
\begin{align*}
|\omega|^2 = \sum_{i_1, \ldots, i_p} \left( \omega_{i_1 \ldots i_p} \right)^2 = 1, \ |S|^2= \sum_{i=1}^n \mu_i^2 = 1, \ \tr(S)=\sum_{i=1}^n \mu_i = 0
\end{align*} 
yields Lagrange multipliers $\alpha_1, \alpha_2, \alpha_3 \in \R$ such that
\begin{align*}
\left( \left( \sum_{j \in \lbrace j_1, \ldots, j_p \rbrace} \mu_j \right)^2 - \alpha_1 \right) \omega_{j_1 \ldots j_p} & = 0, \\
2 \sum_{I=(i_1, \ldots, i_p)}  \left( \sum_{i \in I} \mu_i \right) \chi_I(j) \left( \omega_{i_1 \ldots i_p} \right)^2 - 2 \alpha_2 \mu_j - \alpha_3 & = 0
\end{align*}
for all $j, j_1, \ldots, j_p=1, \ldots, n,$ where for $I=(i_1, \ldots, i_p)$
\begin{align*}
\chi_I(i) = \begin{cases}
1 & \ i \in I, \\
0 & \ i \notin I
\end{cases}
\end{align*}
is the characteristic function.

In particular, if $\omega_{j_1 \ldots j_p} \neq 0,$ then $\left( \sum_{j \in \lbrace j_1, \ldots, j_p \rbrace } \mu_i \right)^2 = \alpha_1$ is constant and thus $| S \omega |^2 = \alpha _1^2.$ Therefore it suffices to show that 
\begin{align*}
\left( \sum_{i=1}^p \mu_i \right)^2 \leq \frac{p(n-p)}{n}
\end{align*}
provided that
\begin{align*}
\sum_{i=1}^n \mu_i^2 = 1 \ \text{ and } \ \sum_{i=1}^n \mu_i = 0.
\end{align*} 
This again yields Lagrange multipliers $\beta_1, \beta_2 \in \R$ such that
\begin{align*}
1 - 2 \beta_1 \mu_j - \beta_2 & = 0 \hspace{6mm} \text{for } 1 \leq j \leq p, \\
- 2 \beta_1 \mu_j - \beta_2 & = 0 \hspace{6mm} \text{for } p+1 \leq j \leq n.
\end{align*}
This implies $\mu_1 = \ldots = \mu_p$ and $\mu_{p+1} = \ldots = \mu_n.$ Solving
\begin{align*}
p \mu_1^2 + (n-p) \mu_n^2 & = 1, \\
p \mu_1 + (n-p) \mu_n & = 0
\end{align*}
yields $\mu_1^2 = \frac{n-p}{pn}$ and $\mu_n^2= \frac{p}{(n-p)n}$ and thus 
\begin{align*}
\left( \sum_{i=1}^p \mu_i \right)^2 = p^2 \mu_1^2 = \frac{p(n-p)}{n}
\end{align*}
as claimed.
\end{proof}

\begin{example}
\label{NormNSOmegaEstimateSharp}
\normalfont
The estimate in lemma \ref{WeightsSecondCurvatureOnForms} (ii) is sharp for $\omega=e^1 \wedge \ldots \wedge e^p$ and $S \in S_0^2(V)$ given by $S(e_i)=\mu_i e_i$ with
\begin{align*}
\mu_1 = \ldots  = \mu_p & = \sqrt{\frac{n-p}{np}}, \\
\mu_{p+1} = \ldots = \mu_n & = - \sqrt{\frac{p}{(n-p)n}}.
\end{align*}
\end{example}

\begin{corollary}
\label{BoundingCurvSecondKind}
If the curvature operator of the second kind is $\frac{n+2}{2}$-nonnegative, then 
\begin{align*}
g( \overline{R}( \omega^{S_0^2} ), \omega^{S_0^2} ) \geq 0.
\end{align*}
\end{corollary}
\begin{proof}
Recall that 
\begin{align*}
g( \overline{R}( \omega^{S_0^2} ), \omega^{S_0^2} ) = \sum_{\alpha} \lambda_{\alpha} |S_{\alpha} \omega |^2,
\end{align*}
where $\lbrace S_{\alpha} \rbrace$ is an orthonormal eigenbasis for $\mathcal{R}$ with corresponding eigenvalues $\lbrace \lambda_{\alpha} \rbrace$. In particular, the total weight is $| \omega^{S_0^2} |^2 $ and highest weight is bounded by $\frac{p(n-p)}{n} | \omega|^2$ due to lemma \ref{WeightsSecondCurvatureOnForms}. Thus,
\begin{align*}
g( \overline{R}( \omega^{S_0^2} ), \omega^{S_0^2} ) \geq [\mathcal{R}, \frac{p(n-p)}{n}, \frac{p(n-p)}{n} \frac{(n+2)}{2}] \cdot | \omega |^2
\end{align*}
and the weight principle \ref{WeightPrinciple} implies the claim.
\end{proof}

X. Li \cite{LiCurvatureOperatorSecondKind} observed that $\Ric \geq \frac{\scal}{n(n+1)} \geq 0$ provided the curvature operator of the second kind is $n$-nonnegative. An application of the Bochner technique hence yields Theorem \ref{NishikawasConjecture}. \vspace{2mm}

\textit{Proof of Theorem \ref{NishikawasConjecture}.} By passing to the orientation double cover if needed, we may assume that $(M,g)$ is oriented. Thus we may assume $p \leq \frac{n}{2}$ due to Poincar\'e duality. Proposition \ref{BochnerFormulaForCurvSecondKind} and corollary \ref{BoundingCurvSecondKind} hence show that $g(\Ric_L(\omega), \omega) \geq 0$, and thus all harmonic forms are parallel. 

If $\omega$ is a parallel $p$-form for $1 \leq p \leq \frac{n}{2}$ and there is $q \in M$ with $\scal_q >0$, then proposition \ref{BochnerFormulaForCurvSecondKind} implies $\scal | \omega |^2 = 0$ at $q$. In particular, $\omega$ vanishes at $q$ and consequently $\omega=0.$ 

Otherwise, $(M,g)$ is scalar flat and hence $\mathcal{R} = 0.$ In particular, $(M,g)$ is flat. $\hfill \Box$

\begin{proposition}
\label{EvalROnTraceFreeTensors}
The trace-free, symmetric $(0,2)$-tensors 
\begin{align*}
\phi_{ij} & = \frac{1}{\sqrt{2}} \left( e^i \otimes e^j + e^j \otimes e^i \right), \ \ 1 \leq i < j \leq n, \\
\psi_{k} & = \frac{1}{\sqrt{(n-k+1)(n-k)}} \left( - k e^k \otimes e^k + \sum_{l=k+1}^n e^l \otimes e^l \right), \ \ k=1, \ldots, n-1,
\end{align*}
form an orthonormal basis for $S_0^2(V)$. Moreover, $g( \overline{R}( \phi_{ij} ), \phi_{ij} ) = R_{ijij}$ and in particular 
\begin{align*}
\sum_{\substack{j=1 \\ j \neq i}}^n g( \overline{R}( \phi_{ij} ), \phi_{ij} ) & = R_{ii}, \\
\sum_{k=1}^p g( \overline{R}( \psi_k ), \psi_k ) & = \frac{2}{n-p} \left( \sum_{k=1}^p R_{kk} - \sum_{1 \leq k < l \leq p} R_{klkl} \right) - \frac{p}{(n-p)n} \scal.
\end{align*}
\end{proposition}
\begin{proof}
This is a straightforward computation.
\end{proof}

\begin{lemma}
\label{WeakEstimatePRicci}
For an algebraic curvature tensor $R$ let $\mathcal{R}$ denote the corresponding curvature operator of the second kind. 

The Ricci tensor satisfies $\Ric \geq \left[ \mathcal{R}, 1, (n-1) \right]$ and for $p \geq 2$ we have
\begin{align*}
 \sum_{i=1}^p R_{ii} \geq \left[ \mathcal{R}, 2, p(n-1) \right]
\end{align*}
with respect to any orthonormal basis $e_1, \ldots, e_n$ for $V.$
\end{lemma}
\begin{proof}
Proposition \ref{EvalROnTraceFreeTensors} implies that
\begin{align*}
\sum_{i=1}^p R_{ii} = 2 \sum_{1 \leq i < j \leq p}  g( \mathcal{R}( \phi_{ij} ), \phi_{ij}) + \sum_{1 \leq i \leq p} \sum_{p+1 \leq j \leq n}  g( \mathcal{R}( \phi_{ij} ), \phi_{ij}).
\end{align*}
Since the $\lbrace \phi_{ij} \rbrace$ are orthonormal, we obtain 
\begin{align*}
\sum_{i=1}^p R_{ii} \geq [\mathcal{R}, 2 , 2 \cdot \frac{p(p-1)}{2}+ p(n-p)] = [\mathcal{R}, 2, p(n-1)].
\end{align*}
\end{proof}

\begin{remark}
\normalfont
The same technique also yields that for $p \geq 2$ the Ricci tensor is $p$-nonnegative provided that any sum of $\frac{p(n-1)}{2}$ sectional curvatures $R_{ijij}$ is nonnegative. Indeed, note that
\begin{align*}
\sum_{i=1}^p R_{ii} = \sum_{i=1}^p \sum_{j=1}^n R_{ijij} = 2 \sum_{1 \leq i < j \leq p} R_{ijij} + \sum_{1 \leq i \leq p} \sum_{p+1 \leq j \leq n} R_{ijij} = [R_{ijij}, 2, p(n-1)]
\end{align*}
and the weight principle \ref{WeightPrinciple} applies.
\end{remark}

\begin{proposition}
\label{WeakEigenvalueEstimate}
If $p \leq \frac{n}{2}$ and $\omega$ is a $p$-form, then
\begin{align*}
\frac{3}{2} g( \Ric_L(\omega), \omega) \geq \left[ \mathcal{R}, \frac{1}{n(n+2)} \left( n^2 p - n p^2 - 2np + 2n^2 + 4n -8p \right), \frac{3}{2}p(n-p) \right] \cdot | \omega |^2.
\end{align*}
\end{proposition}
\begin{proof}
Due to propositions \ref{TracesSecondCurvature} and \ref{BochnerFormulaForCurvSecondKind} and lemmas \ref{ComputationsWithWeights}, \ref{WeightsSecondCurvatureOnForms} and \ref{WeakEstimatePRicci} we have with respect to an orthonormal basis that diagonalizes the Ricci tensor
\begin{align*}
\frac{3}{2} g( \Ric_L(\omega), \omega)  = & \ g( \mathcal{R}(\omega^{S_0^2}), \omega^{S_0^2}) + \frac{n-2p}{n} \sum_{I=(i_1, \ldots, i_p)} \left( \sum_{i \in I} R_{ii} \right) \omega_{I}^2 + \frac{p^2}{n^2} \scal | \omega |^2 \\
\geq & \ \left[\mathcal{R}, \frac{p(n-p)}{n}, \frac{p(n-p)}{n} \cdot \frac{n+2}{2} \right] \cdot | \omega |^2 + \frac{n-2p}{n} [\mathcal{R}, 2, p(n-1)] \cdot | \omega |^2 \\
& \  + \frac{p^2}{n^2} \cdot \frac{2n}{n+2} \left[\mathcal{R}, 1, \frac{(n+2)(n-1)}{2} \right] \cdot | \omega |^2 \\
\geq & \ \left[ \mathcal{R}, \frac{1}{n} \left( p(n-p) + 2(n-2p)+ \frac{2p^2}{n+2} \right), \right. \\
& \hspace{9mm} \left. \frac{p}{n} \left( (n-p) \frac{n+2}{2}+(n-2p)(n-1) + p(n-1) \right) \right] \cdot | \omega |^2.
\end{align*}
Note that the condition $p \leq \frac{n}{2}$ ensures that the $(n-2p)$ factor in the second term is nonnegative.
\end{proof}

By considering both $\lbrace \phi_{ij} \rbrace$ and $\lbrace \psi_k \rbrace$ from proposition \ref{EvalROnTraceFreeTensors}, we can refine lemma \ref{WeakEstimatePRicci}. 

\begin{lemma}
\label{ImprovedEstimatePRicci}
For an algebraic curvature tensor $R$, let $\mathcal{R}$ denote the curvature operator of the second kind. 

With respect to any orthonormal basis, the Ricci tensor satisfies 
\begin{align*}
R_{11} \geq \frac{n-1}{n+1} \left[ \mathcal{R}, 1, n \right] + \frac{1}{n(n+1)} \scal
\end{align*}
and for $p \geq 2$ we have
\begin{align*}
\sum_{i=1}^p R_{ii} \geq \frac{n-p+1}{n-p+2} [ \mathcal{R}, 2, p(n-1) ] + \frac{p}{n(n-p+2)} \scal.
\end{align*}
\end{lemma}
\begin{proof}
Proposition \ref{EvalROnTraceFreeTensors} implies that
\begin{align*}
\sum_{1 \leq k< l \leq p} & g( \mathcal{R}( \phi_{kl} ), \phi_{kl} ) + \sum_{1 \leq k \leq p} \sum_{p+1 \leq l \leq n} g( \mathcal{R}( \phi_{kl} ), \phi_{kl} ) + \sum_{k=1}^p g( \mathcal{R}( \psi_{k} ), \psi_{k} ) = \\
= & \ \frac{n-p+2}{n-p} \left( \sum_{k=1}^p R_{kk} - \sum_{1 \leq k < l \leq p} R_{klkl} \right) - \frac{p}{(n-p)n} \scal.
\end{align*}
Since $\lbrace \phi_{kl} \rbrace \cup \lbrace \psi_k \rbrace$ is orthonormal, the sum is nonnegative provided that $\mathcal{R}$ is $\frac{p}{2}(2n-p+1)$-nonnegative.

Moreover, it follows that 
\begin{align*}
\sum_{i=1}^p R_{ii} = & \ \frac{2(n-p+1)}{n-p+2} \sum_{1 \leq k< l \leq p} R_{klkl} + \frac{n-p}{n-p+2} \left( \sum_{1 \leq k \leq p} \sum_{p+1 \leq l \leq n} R_{klkl}  + \sum_{k=1}^p g( \overline{R}( \psi_k), \psi_k ) \right) \\
& \ + \frac{p}{(n-p+2) n} \scal.
\end{align*}

In particular, with regard to eigenvalues of $\mathcal{R},$ the terms in the first line have highest weight $\frac{2(n-p+1)}{n-p+2}$ and total weight $\frac{(n-p+1)p(n-1)}{n-p+2}.$ 
\end{proof}

\begin{proposition}
\label{ImprovedEigenvalueEstimate}
If $p \leq \frac{n}{2}$ and $\omega$ is a $p$-form, then
\begin{align*}
\frac{3}{2} g( \Ric_L(\omega), \omega) \geq \left[ \mathcal{R}, \frac{1}{n(n+2)} \left( n^2 p - n p^2 - 2np + 2n^2 + 2n -4p \right), \frac{3}{2}p(n-p) \right] \cdot | \omega |^2.
\end{align*}
\end{proposition}
\begin{proof}
The proof is analogous to the proof of proposition \ref{WeakEigenvalueEstimate}. Instead of lemma \ref{WeakEstimatePRicci} one uses lemma \ref{ImprovedEstimatePRicci}.
\end{proof}

\begin{remark}
\label{StrongImprovementPIsOne}
\normalfont
A $1$-form $\omega$ in fact satisfies the slightly improved estimate
\begin{align*}
\frac{3}{2} g( \Ric_L(\omega), \omega) \geq \left[ \mathcal{R}, \frac{2n-1}{n+2}, \frac{3(n-1)}{2} \right] \cdot | \omega |^2.
\end{align*}
\end{remark}

\begin{proposition}
\label{EigenvalueEstimateEinsteinCase}
Let $p \leq \frac{n}{2}$ and let $\omega$ be a $p$-form. If $\Ric = \frac{\scal}{n} g$, then
\begin{align*}
\frac{3}{2} g( \Ric_L(\omega), \omega) \geq \frac{p(n-p)}{n} \left[ \mathcal{R}, \frac{n+4}{n+2}, \frac{3n}{2} \right] \cdot | \omega |^2.
\end{align*}
\end{proposition}
\begin{proof}
According to proposition \ref{BochnerFormulaForCurvSecondKind} we have
\begin{align*}
\frac{3}{2} g( \Ric_L(\omega), \omega) = g( \mathcal{R}(\omega^{S_0^2}), \omega^{S_0^2}) + \frac{p(n-p)}{n^2} \scal | \omega |^2.
\end{align*}
Thus, the weight principle \ref{WeightPrinciple}, lemma \ref{WeightsSecondCurvatureOnForms} and proposition \ref{TracesSecondCurvature} yield
\begin{align*}
\frac{3}{2} g( \Ric_L(\omega), \omega) \geq & \ \frac{p(n-p)}{n} \left[ \mathcal{R}, 1, \frac{n+2}{2} \right] \cdot | \omega |^2 \\
& \ + \frac{p(n-p)}{n^2} \frac{2n}{n+2} \left[\mathcal{R}, 1, \frac{(n-1)(n+2)}{2} \right] \cdot | \omega |^2 \\
\geq & \ \frac{p(n-p)}{n} \left[ \mathcal{R}, \frac{n+4}{n+2}, \frac{3n}{2} \right] \cdot | \omega |^2.
\end{align*}
\end{proof}

\section{Proofs of the main Theorems}
\label{SectionProofs}

In this section we prove Theorems \ref{MainTheoremEinstein} - \ref{WeakEstimation}. The proof of Theorem \ref{NishikawasConjecture} was given after corollary \ref{BoundingCurvSecondKind}. We conclude the section with an example of an algebraic, $(n+1)$-positive curvature operator of the second kind with negative Ricci curvatures, and the example of the rational homology sphere $SU(3)/SO(3).$ \vspace{2mm}

\textit{Proof of Theorem \ref{MainTheoremEinstein}}. Recall that $N=\frac{3n}{2} \frac{n+2}{n+4}.$ 

(c) The fact that all forms are parallel if the curvature operator of the second kind is $N$-nonnegative is a direct consequence of proposition \ref{EigenvalueEstimateEinsteinCase}, the weight principle \ref{WeightPrinciple} and the Bochner technique as outlined at the beginning of section \ref{SectionBochnerTechnique}. 

(a) To obtain vanishing of the Betti numbers, let $\omega$ be a harmonic $p$-form and suppose that the curvature operator of the second kind is $N$-positive. Since $\omega$ is parallel, proposition \ref{EigenvalueEstimateEinsteinCase} yields 
\begin{align*}
0 = \frac{3}{2} g( \Ric_L( \omega ), \omega ) \geq  \frac{p(n-p)}{n} \left[ \mathcal{R}, \frac{n+4}{n+2}, \frac{3n}{2} \right] | \omega |^2.
\end{align*}
The weight principle \ref{WeightPrinciple} implies that $\left[ \mathcal{R}, \frac{n+4}{n+2}, \frac{3n}{2} \right]>0,$ and hence $\omega$ vanishes.

(b) If $\mathcal{R}$ is $N'$-nonnegative for some $N'<N,$ then by the weight principle \ref{WeightPrinciple} (d), either $\mathcal{R}$ is $N$-positive and $M$ is a rational homology sphere by part (a), or $\mathcal{R}$ is $1$-nonnegative and $(M,g)$ is either flat or a rational homology sphere by Theorem \ref{NishikawasConjecture}. $\hfill \Box$

\vspace{2mm}

{\em Proof of Theorem \ref{BettiNumbersCurvatureSecondKind}}. The proof is analogous to the proof of Theorem \ref{MainTheoremEinstein}. Instead of proposition \ref{EigenvalueEstimateEinsteinCase} one uses \ref{ImprovedEigenvalueEstimate}. $\hfill \Box$

\vspace{2mm}

\textit{Proof of Theorem \ref{WeakEstimation}}. The weight principle \ref{WeightPrinciple} and the estimate $\Ric \geq \left[ \mathcal{R}, 1, (n-1) \right]$ in proposition \ref{WeakEstimatePRicci} immediately imply that $\Ric \geq (n-1) \kappa$ provided that the average of the lowest $(n-1)$ eigenvalues of the curvature operator of the second kind is bounded from below by $\kappa.$ The methods of Gallot and P.Li imply Theorem \ref{WeakEstimation}, cf. \cite[Theorem 1.9]{PetersenWinkNewCurvatureConditionsBochner}.  $\hfill \Box$

\begin{remark}
\normalfont
Note that proposition \ref{ImprovedEstimatePRicci} in fact provides a lower bound on Ricci curvature if the average of the lowest $n$ eigenvalues of $\mathcal{R}$ is bounded from below by $\kappa.$ 

If the eigenvalues $\lambda_1 \leq \ldots \leq \lambda_N$ of the curvature operator of the second kind satisfy $\lambda_1 + \ldots + \lambda_m \geq m \kappa,$ then $\lambda_j \geq \kappa$ for $j > m$ and thus
\begin{align*}
\scal = \frac{2n}{n+2} \tr ( \mathcal{R} ) \geq \frac{2n}{n+2} \kappa \cdot \dim ( S_0^2(TM) ) = n(n-1) \kappa.
\end{align*}
In particular, if in addition $(M,g)$ is Einstein, then $\Ric = \frac{\scal}{n} g \geq (n-1) \kappa g$ and we obtain the estimation theorem corresponding to Theorem \ref{MainTheoremEinstein} from the weight principle \ref{WeightPrinciple} and the work of Gallot and P.Li as before.

The proof of the estimation theorem corresponding to Theorem \ref{BettiNumbersCurvatureSecondKind} is analogous, provided a lower bound on the Ricci curvature is assumed explicitly, cf. examples \ref{ExplicitValuesGeneralBettinumberTheorem} and \ref{NoControlOnRic}.
\end{remark}

\begin{example}
\label{ExplicitValuesGeneralBettinumberTheorem}
\normalfont
Let $\mathcal{S}= \frac{3}{2}p(n-p)$. Let
\begin{align*}
\Omega^{\text{pre}} = \frac{1}{n(n+2)} \left( n^2 p - n p^2 - 2np + 2n^2 + 4n -8p \right)
\end{align*}
be the highest weight obtained with the preliminary estimate in proposition \ref{WeakEigenvalueEstimate} and let
\begin{align*}
\Omega= \frac{1}{n(n+2)} \left( n^2 p - n p^2 - 2np + 2n^2 + 2n -4p \right)
\end{align*}
be the highest weight in proposition \ref{ImprovedEigenvalueEstimate}. 

The difference of the highest weights is $\Omega^{\text{pre}} - \Omega = \frac{2(n-2p)}{n(n+2)}$.

Furthermore, the quotient $\frac{\mathcal{S}}{\Omega}$ is increasing in $p$.
Note that we require $p \leq \frac{n}{2}$ in the estimates above due to the $(n-2p)$ factor of the Ricci curvature term. Thus we get the weakest curvature condition for $p=\frac{n}{2}.$ 

For $p=2$ we obtain 
\begin{align*}
\frac{\mathcal{S}}{\Omega} = \frac{3n}{4} \frac{n^2-4}{n^2-\frac{3}{2}n -2}.
\end{align*}

For $p=4$ and $n \geq 4$ we have
\begin{align*}
\frac{\mathcal{S}}{\Omega} = n \frac{n^2-2n-8}{n^2-\frac{11}{3}n -\frac{8}{3}} > n.
\end{align*}

For $p=5$ and $n \geq 5$ we have 
\begin{align*}
\frac{\mathcal{S}}{\Omega} = \frac{15 n}{14} \frac{n^2-3n-10} {n^2-\frac{33}{7}n -\frac{20}{7}} > \frac{15 n}{14}.
\end{align*}

For $p= \frac{n}{2}$, we have 
\begin{align*}
\frac{\mathcal{S}}{\Omega} = \frac{3n}{2} \frac{n+2}{n+4}
\end{align*}
as in the Einstein case. Furthermore, note that for any fixed $p$ we have
\begin{align*}
\lim_{n \to \infty}  \frac{\mathcal{S}}{n \cdot \Omega} = \frac{3p}{2(p+2)}.
\end{align*}
\end{example}

Recall that X.Li \cite{LiCurvatureOperatorSecondKind} proved the lower Ricci curvature bound $\Ric \geq \frac{\scal}{n(n+1)} \geq 0$, provided the curvature operator of the second kind is $n$-nonnegative. In contrast, example \ref{NoControlOnRic} below exhibits an $(n+1)$-positive curvature operator of the second kind with negative Ricci curvatures. 

In particular, for $p=5, \ldots, \frac{n}{2}$, our curvature conditions do {\em not} imply nonnegative Ricci curvature, while we are still able to control the Betti numbers. For example, as a special case of Theorem \ref{BettiNumbersCurvatureSecondKind}, we have

\begin{corollary}
\label{BettiNumberControlWithoutRic}
Let $(M,g)$ be a compact $n$-dimensional Riemannian manifold. Let $n \geq 14$ and $5 \leq p \leq n-5.$ 

If the curvature operator of the second kind is $(n+1)$-nonnegative, then all harmonic $p$-forms are parallel.

If in addition $\scal>0$ at a point in $M$, then the $p$-th Betti number $b_p(M, \R)$ vanishes.
\end{corollary}

\begin{example}
\label{NoControlOnRic}
\normalfont
In \cite{LiCurvatureOperatorSecondKind}, X.Li observed that the curvature operator of the second kind of $S^1 \times S^{n-1}$ has the eigenvalues $-\frac{n-2}{n}$ with multiplicity one, $0$ with multiplicity $n-1$ and $1$ with multiplicity $\frac{(n-2)(n+1)}{2}.$ In particular, the curvature operator of the second kind is $(n+1)$-positive, but not $n$-nonnegative. For small $\kappa<0$, we obtain an algebraic, $(n+1)$-positive curvature operator of the second kind with Ricci curvature $R_{11}<0$ by adding the curvature tensor $\frac{\kappa}{2} g \owedge g$ of constant sectional curvature $\kappa$ to the curvature tensor of $S^1 \times S^{n-1}$.
\end{example}

\begin{example}
\label{CurvatureSU3SO3}
\normalfont
Consider the irreducible symmetric space $M=SU(3)/SO(3)$. As Wolf \cite{WolfSymmetricRealCohomSpheres} observed, $M$ is a rational homology sphere. However, note that $H_2(M,\Z)=\Z/2\Z.$ 

The curvature operator $\mathfrak{R} \colon \Ext^2TM \to\Ext^2TM$ has a $7$-dimensional kernel and the nonzero eigenvalue $5/2$ with multiplicity $3$. In particular, the Ricci tensor satisfies $\Ric = 3 g.$

The curvature operator of the second kind  $\mathcal{R} \colon S_0^2(TM) \to S_0^2(TM)$ has eigenvalues $-3/2$ with multiplicity $5$ and $2$ with multiplicity $9$. In particular, it is $9$-positive but not $8$-nonnegative.

\end{example}


\end{document}